\documentclass[reqno]{amsart}
\usepackage{hyperref}

\newcommand{\PP}{{\mathbb P}}
\newcommand{\NN}{{\mathsf N}}

\begin{document}
\title[
Return of $k$-bonacci random walks ] {Return of $k$-bonacci random
walks}

\author[N. Attia, C. Souissi]
{Najmeddine Attia, Chouha\"{i}d Souissi}  

\address{Najmeddine Attia, Chouha\"{i}d Souissi \newline
Department of Mathematics, Faculty of Sciences of Monastir,
University of Monastir, 5000 -Monastir, Tunisia}
\email{najmeddine.attia@gmail.com, chsouissi@yahoo.fr}



\begin{abstract}
In this work, the probability of return for random walks on
$\mathbb{Z}$, whose increment is given by the $k$-bonacci
sequence, is determined. Also, the Hausorff, packing and
box-counting dimensions of the set of these walks that return an
infinite number of times to the origin are given. As an
application, we study the return for tribonacci random walks to
the first term of the tribonacci sequence.
\end{abstract}

\maketitle \numberwithin{equation}{section}
\newtheorem{theorem}{Theorem}[section]
\newtheorem{lemma}[theorem]{Lemma}
\newtheorem{remark}[theorem]{Remark}
\newtheorem{example}[theorem]{Example}
\newtheorem{examples}[theorem]{Examples}
\newtheorem{definition}[theorem]{Definition}
\newtheorem{proposition}[theorem]{Proposition}
\newtheorem{corollary}[theorem]{Corollary}
\allowdisplaybreaks

\section{Introduction and main results}

The Fibonacci sequence, commonly denoted by $(f_n)_{n\geq0}$, is a
sequence of integer numbers such that each of them is the sum of
the two preceding ones, starting from $0$ and $1$, i.e.,
 $$
\begin{cases}
f_{0}=0,\quad f_{1}=1, & \\
f_{n}=f_{n-1}+f_{n-2}, \qquad for\quad n\geq 2.
\end{cases}
 $$

This sequence was first introduced by Leonardo Fibonacci and is
tightly connected to the golden ratio $\varphi =(1+\sqrt{5})/2
=1.61803398\ldots$. Since then, many researchers has been
interested to the study of the properties of this sequence and
their applications. One can cite, for example, \cite{Mak} where it
was proved that $(f_n)_{n\geq0}$ increases exponentially with $n$
at a rate given by $\varphi$. A more general case was treated in
\cite{Vis}, where the author considered the random Fibonacci
sequence $(t_n)_{n\geq0}$ defined by $t_1 =t_2=1$ and for $n>2$,
 $$
t_n=\pm t_{n-1} \pm t_{n-2},
 $$
where each $\pm$ sign is independent and either $+$ or $-$ with
probability $1/2$ . It is not even obvious that $|t_n|$ should
increase with $n$. Although, it was proved that, almost surely
$$
\lim_{n\to+\infty}\sqrt[n]{|t_n|} = 1,13198824\ldots
$$
Later, in \cite{Neu}, the author considered the Fibonacci random
walk and determined the probability of its return to the origin.
More precisely, he considered the random walk on $\mathbb{Z}$
whose increments are given by $(f_n)_{n\geq0}$, i.e.,
 $$
\hat F_n =\sum_{k=1}^n f_i w_i,
 $$
where $(w_i)_{i\geq1}$ is a sequence of independent, identically
distributed random variables with $\PP(w_1=\pm 1)=1/2$. \\

We denote by $\mathbb{N}$ the set of positive integers and
consider the space of infinite sequences $\mathcal{A} =\{-1,
1\}^\mathbb{N}$. We define, for an elementary event
$w\in\mathcal{A}$, the set
 $$
F(w)=\big\{n\geq1, \; \hat F_n(w)=0\big\}.
 $$
Denoting by $"\sharp B"$ the cardinality of a given set $B$, we
set,
 $$
R_i =\big\{ w\in\mathcal{A} \;|\;\sharp F(w)=i\big\},\qquad
i\in\mathbb{N}\cup\{0\}
 $$
and
 $$
\NN=\Big\{ w\in\mathcal{A} \;|\;\sharp F(w)=\infty\Big\}.
 $$
It is known from \cite{Neu} that the probability of $R_i$ is
$3/4^{i+1}$. In particular, $\PP(\NN)=0$. The idea of studying
such kind of problems comes from a classical result due to Polya
\cite{Pol}, who was interested to the simple random walk of length
$n\geq1$, on the integers, given by
 $$
S_n =\displaystyle\sum_{j=1}^n w_j.
 $$
This means that $S_n$ is seen as to be the position, after $n$
steps, of a walk on the integers of an individual, who is supposed
to start its motion at the origin of the lattice, i.e.,$
S_0=w_0=0.$ P\'olya \cite{Pol} showed this walk to be recurrent,
which means that it almost surely returns to the origin in a
finite number of steps. The reader can also see, for example,
\cite{Fel, Kes, MaW} for more discussions on this problem.\\

In this paper, we are interested, for a given integer $k\geq2$, to
the $k$-bonacci random walk given by $\hat F_n$, where we take in
consideration the $k$-bonacci sequence $(f_n)_{n\geq0}$ defined by
$f_0=0$,
\begin{equation} \label{fn}
f_n= \displaystyle \sum_{j=1}^{k}f_{n-j}, \qquad\text{for}\quad
n\geq k+1.
\end{equation}
and the $k$ initializing terms $(f_n)_{1\leq n\leq k}$ are
supposed to satisfy the following condition
 $$
\displaystyle \sum_{j=1}^n\pm f_j \neq0\qquad\text{for } 1\leq
n\leq k, \eqno{(AS)}
 $$
The condition $(AS)$ is guarantied, for example, in the following
situation
 $$
f_n=1+\displaystyle\sum_{j=0}^{n-1} f_j, \qquad\text{for}\quad
1\leq n\leq k.
 $$
We study the probability of return of the $k$-bonacci random walk
to the origin. For this, we establish a necessary and sufficient
condition for the $k$-bonacci random walk to reach $0$ at least
one time (Proposition \ref{f0}). Our first main result is the
following.
\begin{theorem}\label{2}
For $i\in \mathbb{N},\quad\PP(R_i)
=\displaystyle\frac1{2^{(k+1)i}}$.
\end{theorem}

Next, we are interested to the set $\NN$ of walks returning
infinitely many times to 0. Since  $\PP(\NN)=0$, it's natural to
ask the question about the size of $\NN$ as a subset of
$\mathcal{A}$. Denoting by $\dim_H, \dim_P$ and $\dim_B$
respectively the Hausdorff, packing and box-counting dimensions,
we can state our second result as follows.
\begin{theorem}\label{3}
$\dim_H(\NN) =\dim_P(\NN) =\dim_B(\NN) =\displaystyle\frac1{k+1}.$
\end{theorem}

We refer the reader to the Appendix A and the reference therein
for the definitions and more details about these fractal
dimensions.

As an application, using the same technics, we study the return
for the tribonacci random walks to the term $f_1$ of the
tribonacci sequence.

\section{Proof of Main results}

We consider the $k$-bonacci sequence $(f_i)_{i\geq0}$ defined by
(\ref{fn}) and $(AS)$. Let $n\geq k+1$. We easily obtain by
induction,
\begin{equation}\label{n+2}
\displaystyle\sum_{j=1}^nf_j< f_{n+2}.
\end{equation}
Moreover,
 $$
\begin{array}{ll} f_{n+1} & =\displaystyle\sum_{j=n-k+1}^nf_{j}
=f_n +\displaystyle\sum_{j=n-k+1}^{n-1}f_{j} \\ & =f_n
+\displaystyle\sum_{j=n-k}^{n-1}f_{j} -f_{n-k} \\ & =2f_n
-f_{n-k}.
\end{array}
 $$
This means that
\begin{equation}\label{n+1}
2f_n \geq1+f_{n+1}.
\end{equation}

Next, we give a necessary and sufficient condition to obtain
$\sum_{i=1}^nw_if_i=0$. For this, we consider, for any integer
$i\geq 2$, the finite sequences,
 $$
\begin{array}{lclclcc} v_{i+}= & \underbrace{+1,\;+1,\;\cdots,\;+1,} & -1 &
\text{ and } \qquad v_{i-}= & \underbrace{-1, \;-1, \;\cdots,\;-1,} & +1. \\
& i \text{ times} &&& i \text{ times} & \end{array}
 $$

We also consider, for a given integer $p\geq1$, the condition
 $$
(w_{p}, \; w_{p+1},\; \ldots, \; w_{p+k}) \in\{ v_{k+}, v_{k-} \}
\eqno{C(k,p)}.
 $$
It is clear that if the condition $C(k,p)$ holds, then
$\displaystyle\sum_{j=p}^{p+k}w_jf_j=0$.

\begin{proposition}\label{f0}
Consider the $k$-bonacci sequence $(f_i)_{i\geq0}$ given by $(AS)$
and (\ref{fn}) and let $w=(w_i)_{i\geq0}\in\mathcal{A}$. We have,
\begin{equation}\label{sum0} \hat F_n(w)=0
\end{equation}
if and only if $n=(k+1)m$, for some integer $m>0,$ and
\begin{equation}\label{wn0>k}
C\big(k,(k+1)i+1\big)\quad\text{holds for all}\quad 0\leq i<m.
\end{equation}
\end{proposition}

To prove this proposition, we need to show the following result.
\begin{lemma}\label{pp}
Suppose that $C(k,p)$ does not hold for an integer $p\geq1$, then
\begin{enumerate}
\item $\Big|\displaystyle \sum_{j=p}^{p+k}w_jf_j\Big|\geq 2f_{p}.$

\item $|\hat F_{p+k}(w)|>1.$
\end{enumerate}
\end{lemma}

\begin{proof}
\begin{enumerate}
\item By (\ref{fn}), we have
\begin{equation}\label{eqp}
\begin{array}{ll}
\Big|\displaystyle\sum_{j=p}^{p+k} w_{j}f_{j}\Big| &
=\Big|\displaystyle\sum_{j=p}^{p+k-1}(w_{j}+w_{p+k})f_{j}\Big|.
\end{array}
\end{equation}

We supposes that $w_{p+k}=1$ (the case $w_{p+k}=-1$ is analogous).

We consider the set
 $$
A_{p,k}=\{ j,\quad p\leq j\leq p+k-1,\quad w_{j}+w_{p+k}\neq0\}.
 $$

Since $C(k,p)$ does not hold, then $ A_{p,k}\neq \emptyset.$ So,
equation (\ref{eqp}) leads to
 $$
\Big|\displaystyle\sum_{j=p}^{p+k} w_{j}f_{j}\Big|
=2\displaystyle\sum_{j\in A_{p,k}}f_j \geq 2f_p.
 $$

\item If $p=1$ and $C(k,1)$ does not hold, then using Lemma
\ref{pp} (1), we obtain
 $$
|\hat F_{k+1}(w)| \geq 2f_1>1.
 $$
Otherwise,
 $$
|\hat F_{p+k}(w)| \geq \Big| \displaystyle
\sum_{j=p}^{p+k}w_jf_j\Big|-|\hat F_{p-1}(w)|.
 $$
Since $C(k,p)$ does not hold, then again using Lemma \ref{pp} (1)
leads to
 $$
|\hat F_{p+k}(w)| =2f_{p}-|\hat F_{p-1}(w)| \geq 2f_{p}
-\displaystyle \sum_{j=1}^{p-1} f_j.
 $$
Applying successively \eqref{n+2} and \eqref{n+1}, we obtain
 $$
|\hat F_{p+k}(w)| >2f_p-f_{p+1}\geq 1.
 $$
\end{enumerate}
\end{proof}

\subsection*{Proof of Proposition \ref{f0}}

$\Leftarrow$: Obviously, if (\ref{wn0>k}) is realized, then by
(\ref{fn}) we obtain (\ref{sum0}).

 \vskip0.05in
$\Rightarrow$: Conversely, suppose that (\ref{sum0}) is insured.
Then, using condition $(AS)$, it becomes that $n\geq k+1$. Let $m$
be the unique positive integer such that $n=(k+1)m+t(n)$.

\begin{enumerate}
\item if there exits $p\in\{(k+1)j+t(n), 0\leq j<m\}$, such that
$C(k,p)$ is not satisfied, then we set
 $$
p(n)=\sup\big\{p=(k+1)j+t(n)+1, \quad 0\leq j< m, \quad C(k,p)
\text{ does not hold}\big\}.
 $$
Thanks to Lemma \ref{pp}, we have $ \hat F_n(w) =\hat
F_{p(n)+k}(w)\neq0.$

\item if $t(n)\neq0$ and $C(k,p)$ is satisfied for all $p\in\{
(k+1)j+t(n), 0\leq j<m\}$, then by condition $(AS)$, we have
 $$
\begin{array}{lcl}
\hat F_n(w) =\hat F_{t(n)}(w) +\displaystyle \sum_{i=0}^{m-1}
\Big( & \underbrace{\displaystyle\sum_{t=1}^{k+1} w_{(k+1)i
+t(n)+t}f_{(k+1)i+t(n)+t}} & \Big) =\hat F_{t(n)}(w) \neq0.
\\ & 0 &
\end{array}
 $$
\end{enumerate}

Consequently, we mast have that $t(n)=0$ and \eqref{wn0>k}
satisfied. This ends the proof.

\subsection{Proof of Theorem \ref{2}}

Let $i\geq1$, from Proposition \ref{f0} we deduce that $\hat F_n$
reaches the origin exactly $i$-times if and only if $n\geq
(k+1)i$, with
 $$
\hat F_{(k+1)i}=0\qquad\text{and}\qquad\hat F_{(k+1)(i+1)}\neq 0.
 $$

Moreover, we have
 $$
\PP\big(\hat F_{(k+1)i}=0\big) =\displaystyle\frac{2^i}
{2^{(k+1)i}} =\displaystyle\frac1{2^{ki}}
 $$
and
 $$
\PP\Big(\hat F_{(k+1)(i+1)} =0\Big/\hat F_{(k+1)i}=0\Big) =
\displaystyle\frac2{2^{k+1}} =\displaystyle\frac1{2^k}.
 $$

It follows that
 $$
\begin{array}{ll}
\PP(R_i) & = \PP\Big(\hat F_{(k+1)(i+1)} \neq0\text{ and }\hat
F_{(k+1)i}=0\Big) \\ \\ & = \PP\Big(\hat F_{(k+1)(i+1)}
\neq0\Big/\hat F_{(k+1)i}=0\Big) \times \PP\Big(\hat
F_{(k+1)i}=0\Big) \\ \\ & =\displaystyle\frac{2^k-1}{2^{k(i+1)}}.
\end{array}
 $$

\subsection{Proof of Theorem \ref{3}}

We consider the metric $d$, defined for any couple
$\Big((u_i)_i,(v_i)_i\Big)$ of $\mathcal{A}\times \mathcal{A}$, by
 $$
d\big((u_i)_i,(v_i)_i\big) =\displaystyle\sum_{i=1}^\infty
\displaystyle\frac{|u_i-v_i|}{2^i}.
 $$

Endowed with this metric, $\big(\mathcal{A},d\big)$ becomes a
compact metric space. As a direct consequence of Proposition
\ref{f0}, we obtain the following lemma.

\begin{lemma}\label{self}
We have, $ \qquad\NN=\{v_{k+},v_{k-} \}^{\mathbb N}.$
\end{lemma}

 \bigskip
Now, we consider the mappings $T_1$ and $T_2$ defined for any
$w=(w_i)_i\in \mathcal{A}$, by
 $$
T_1(w) =( v_{k+},\; w_1,\; w_2,\; \ldots) \qquad\text{and}\qquad
T_2(w) =( v_{k-},\; w_1,\; w_2,\; \ldots).
 $$

For $i\in\{1,2\}$ and for any $(u,v)=\Big((u_j)_j,(v_j)_j\Big)\in
\mathcal{A}^2$, we have
 $$
d\big(T_i(u),T_i(v)\big) =\displaystyle\sum_{j=1}^\infty
\displaystyle\frac{\big|\big(T_i(u)\big)_j-\big(T_i(v)\big)_j
\big|}{2^j} =\displaystyle\sum_{j=k+2}^\infty
\displaystyle\frac{|u_j-v_j|}{2^j}
=\displaystyle\frac1{2^{k+1}}\; d(u,v).
 $$

This means that $T_1$ and $T_2$ are contracting similarities on
the metric space $(\mathcal{A}, d)$, with contraction rates
 $$
r_1=r_2=\displaystyle\frac 1{2^{k+1}}.
 $$

Coming buck to \cite{Hut}, one deduces the existence of a unique
compact self-similar subset $F$ of $\mathcal{A}$, such that $
F=T_1(F)\cup T_2(F).$ Lemma \ref{self} implies that $F=N$.
Furthermore, we have $ T_1(N) \cap T_2(N) =\emptyset.$ Hence, $N$
is a self-similar set satisfying the open set condition in the
compact metric space $(\mathcal{A}, d)$. Their fractal dimension
is then given by Corollary \ref{dimension},
 $$
\dim_H(\NN) =\dim_P(\NN) =\dim_B(\NN) =\displaystyle\frac{\ln(2)}
{\ln(2^{k+1})} =\displaystyle\frac1{k+1}.
 $$

\section{Application}\label{app}

We are interested in applying the ideas presented in the previous
sections to a class of tribonacci sequences, defined by $f_0 =0$,
$f_1=1$, $f_2=3$, $f_3=6$ and
\begin{equation} \label{trifn}
f_i =\displaystyle \sum_{j=1}^{3}f_{i-j},\quad \text{for} \quad
i\geq 4.
\end{equation}

The return point on which we focus our attention is no longer the
origin. Our ideas are still applicable when considering the number
of visits of $\hat F_n$ to $f_1$. We consider, for an elementary
event $w\in\mathcal{A}$, the set $F_1(w)$ for which $\hat F_n(w)$
reaches $f_1$ after $n$ steps of the walk, i.e.,
 $$
F_1(w)=\{n\geq1, \; \hat F_n(w)=f_1\}.
 $$

For $i\in \mathbb{N}$, we denote by $R_{1,i}$ the event for which
elements $\hat F_n$ passes through $f_1$ exactly $i$ times, i.e.,
 $$
R_{1,i} =\Big\{ w\in\mathcal{A} \;|\;\sharp F_1(w)=i\Big\}.
 $$

\begin{theorem}\label{41}
For $i\in \mathbb{N}, \quad \PP(R_{1,i})
=\displaystyle\frac7{2^{3(i+1)+1}}$.
\end{theorem}

In a similar way, we consider the set $\NN_1$ consisting on
elements of $\mathcal{A}$ for which $\hat F_n$ passes through
$f_1$ an infinite number of times, i.e.,
 $$
\NN_1=\Big\{ w\in\mathcal{A} \;|\;\sharp F_1(w)=\infty\Big\}.
 $$

\begin{theorem}\label{43}
We have,
 $$
\dim_H(\NN_1)=\dim_P(\NN_1)=\dim_B(\NN_1) =\displaystyle\frac14.
 $$
\end{theorem}

In the same spirit of Proposition \ref{f0}, we have

\begin{proposition}\label{44}
Consider the tribonacci sequence $(f_i)_{i\geq0}$ given by
(\ref{trifn}) and let $(w_i)_{i\geq0}\in\mathcal{A}$, with
$w_1=1$. We have,
\begin{equation}\label{trisumf1} \hat F_n(w)= f_1
\end{equation}
if and only if either $n=1$ or $n=4m+1$, for some integer $m\geq
1$, and
\begin{equation}\label{wn1>k}
C(3,4j+2)\quad\text{ holds for all }\quad 0\leq j<m.
\end{equation}
\end{proposition}

\begin{proof}

$\Leftarrow$: Obviously, if either $n=1$ or $n=4m+1$, for some
integer $m\geq 1$, and \eqref{wn1>k} holds, then thanks to
(\ref{trifn}) we obtain (\ref{trisumf1}).

 \medskip
$\Rightarrow$: Conversely, supposing that $n\geq4$ and that
$C(3,n-p )$ is not satisfied, we obtain a contradiction by Lemma
\ref{pp}. Otherwise, arguing by induction, we obtain
 $$
\hat F_n(w) =\hat F_{t(n)}(w).
 $$
If $t(n)\neq1$, then $\hat F_{t(n)}$ is even and positive, so
$\;|\hat F_{t(n)}|> f_1.$ Again a contradiction.

\noindent It follows that either $n=1$ or $t(n)=1$ and
\eqref{wn1>k} holds.
\end{proof}

\subsection{Proof of Theorem \ref{41}}

We take $i\geq1$. From Proposition \ref{44}, we deduce:

 \bigskip
$\hat F_n$ reaches $f_1$ exactly $i$-times if and only if $n\geq
4i+1$, with
 $$
\hat F_{4i+1}=1\qquad\text{and}\qquad\hat F_{4(i+1)+1}\neq 1.
 $$

Moreover, we have
 $$
\PP\big(\hat F_{4i+1}=1\big) =\displaystyle\frac{2^i}{2^{4i+1}}
=\displaystyle\frac1{2^{3i+1}} \quad\text{and}\quad \PP\Big(\hat
F_{4(i+1)+1} =1\Big/\hat F_{4i+1}=1\Big) =
\displaystyle\frac2{2^{4}} =\displaystyle\frac18.
 $$

It follows that
 $$
\PP(R_{1,i}) = \PP\Big(\hat F_{4(i+1)+1} \neq1\Big/ \hat
F_{4i+1}=1\Big) \times \PP\Big(\hat F_{4i+1}=1\Big)
=\displaystyle\frac7{2^{3(i+1)+1}}.
 $$

\subsection{Proof of Theorem \ref{43}}

We have that
 $$
\NN_1=\{1\}\times\{v_{3+},\; v_{3-}\}^{\mathbb N}.
 $$
We consider the mapping $T$ defined for any $w=(w_i)_i\in
\mathcal{A}$, by
 $$
T(w) =1,\; w.
 $$
For $(u,v)= \Big((u_j)_j,(v_j)_j\Big)\in \mathcal{A}^2$, we have
 $$
d\big(T(u),T(v)\big) =\displaystyle\sum_{j=1}^\infty
\displaystyle\frac{|T(u)_j-T(v)_j|}{2^j}
=\displaystyle\sum_{j=2}^\infty
\displaystyle\frac{|u_{j-1}-v_{j-1}|}{2^j} =\displaystyle\frac12\;
d(u,v).
 $$
This means that $T$ is a bi-Lipschitz mapping. Since
$\NN_1=T(\NN)$, we have
 $$
\dim_H(\NN_1) =\dim_H(\NN) =\displaystyle\frac14.
 $$
Coming buck to Corollary \ref{dimension}, we deduce
 $$
\dim_H(\NN_1) =\dim_B(\NN_1) =\dim_P(\NN_1) =\displaystyle\frac14.
 $$

\section{Concluding remarks and perspectives}

We think it to be very interesting to make the point on some
remarks and possible extensions of our work.

\begin{enumerate}
\item The results given by Theorems \ref{41} and \ref{43} remain
still valid if we take $w_1=-1$. In other words, if we take the
tribonacci sequence defined by (\ref{trifn}) and consider the set
 $$
F_{-1}(w)=\{n\geq1, \; \hat F_n(w)=-f_1\}
 $$
and if we denote, for $i\in \mathbb{N}$, by $R_{-1,i}$ the event
for which elements $\hat F_n$ passes through $(-f_1)$ exactly $i$
times, i.e.,
 $$
R_{-1,i} =\Big\{ w\in\mathcal{A} \;|\;\sharp F_{-1}(w)=i\Big\},
 $$
then, we have that $\PP(R_{-1,i})
=\displaystyle\frac7{2^{3(i+1)+1}}$.

Moreover, if the set $\NN_{-1}$ consists on the elements of
$\mathcal{A}$ for which $\hat F_n$ passes through $(-f_1)$ an
infinite number of times, i.e.,
 $$
\NN_{-1}=\Big\{w\in\mathcal{A}\;|\;\sharp F_{-1}(w)=\infty\Big\},
 $$
then, $\NN_{-1}$ is such that $\dim_H(\NN_{-1}) =\dim_P(\NN_{-1})
=\dim_B(\NN_{-1}) =\displaystyle\frac14.$

\item The results obtained in this work depend strongly on the $k$
initializing terms of the $k$-bonacci sequence: $(f_i)_{1\leq
i\leq k}$. Particularly, thanks to $(AS)$, $\hat F_n$ is allowed
to visit 0 or $\pm f_1$ only one time in its first $k$ steps of
the walk : $(\hat F_i)_{1\leq i\leq k}$. If $(AS)$ is no longer
satisfied, then $\hat F_n$ can reach these values (0 or $\pm f_1$)
more than one time before its $(k+1)$th term: $\hat F_{k+1}$.

\item One can think of the possibility to reach other terms the
$k$-bonacci sequence and the eventual necessary and/or sufficient
conditions to realize this task by the $\hat F_n$.
\end{enumerate}

\begin{appendix}

\section{Fractal dimensions}

For a non-empty subset $U$ of the Euclidian space $\mathbb{R}^m,
m\geq1$, the diameter of $U$ is defined as
 $$
|U|=\sup\{|x-y|, x,y \in U\}.
 $$

Let $I$ and $F$ be respectively nonempty subsets of $\mathbb{N}$
and $\mathbb{R}^n$ ($I$ may be either finite or countable). We say
that $(U_i)_{i\in I}$ is a $\delta$-covering of $F$ if
 $$
F\subset\displaystyle\bigcup_{i\in I}U_i \qquad \text{and}\qquad
0<|U_i|\leq\delta,\quad\forall i\in I.
 $$

The $s$-dimensional Hausdorff measure of $F$ is defined as
 $$
\mathcal{H}^s(F) =\displaystyle\lim_{\delta\to0^+}\inf\Big\{
\displaystyle\sum_{i\in\mathbb{N}}|U_i|^s\Big\},
 $$
where the infimum is taken over all the countable
$\delta$-coverings $(U_i)_{i\in\mathbb{N}}$ of $F$.

The Hausdorff dimension of $F$ is defined as
 $$
\dim_H F=\sup\{s>0, \; \mathcal{H}^s(F)=\infty\} =\inf\{s>0, \;
\mathcal{H}^s(F)=0\},
 $$
with the convention $\sup\emptyset=0$ and $\inf\emptyset=\infty$.

The $s$-dimensional packing measure of $F$ is defined as
 $$
\mathcal{P}^s(F) =\displaystyle\lim_{\delta\to0^+}\sup\Big\{
\displaystyle\sum_{i\in\mathbb{N}}|B_i|^s\Big\},
 $$
where the supremum is taken over all the packings $\{B_i\}_{i\in
\mathbb{N}}$ of $F$ by balls centered on $F$ and with diameter
smaller than or equal to $\delta$. The packing dimension of $F$ is
defined as
 $$
\dim_P(F)=\sup\{s>0, \; \mathcal{P}^s(F)=\infty\} =\inf\{s>0, \;
\mathcal{P}^s(F)=0\},
 $$
with the convention $\sup\emptyset=0$ and $\inf\emptyset=\infty$.

Let $N_\delta (F)$ be the smallest number of sets of diameter at
most $\delta$ which can cover $F$. The lower and upper box
counting dimensions of $F$ are respectively defined as
 $$
\underline{\dim}_B(F) =\lim\inf_{\delta\to 0}
\displaystyle\frac{log\; N_\delta (F)}{log(\delta)},
 $$
and
 $$
\overline{\dim}_B(F) =\lim\sup_{\delta\to 0}
\displaystyle\frac{log\; N_\delta (F)}{log(\delta)}.
 $$

If $\underline{\dim}_B(F) =\overline{\dim}_B(F)$, this common
value is denoted $\dim_B(F)$ and referred to as box counting
dimension or simply the box dimension of $F$, i.e.,
 $$
\dim_B(F) =\lim_{\delta\to 0} \displaystyle\frac{log\; N_\delta
(F)}{log(\delta)}.
 $$

For more details, the reader can be referred, for example, to
\cite{Fal, Mat, Pes}.

\section{Fractal dimension of iterated function system (IFS)}

Let $m$ and $p$ be two positive integers, with $p\geq2$, and $X$
be aa nonempty closed set of $\mathbb{R}^m$. A family $\{S_i, i=1,
\ldots, p\}$ of contractive mappings on $X$ is called an iterated
function system (IFS) on $X$ \cite{Bar}. Hutchinson showed in
\cite{Hut} that there is a unique nonempty compact set $K$ of $X$,
called the attractor of $\{S_i, i=1, \ldots, p\}$, such that
 $$
K=\displaystyle\bigcup_{i=1}^pS_i(K).
 $$

The local dimension at a point $x\in\mathbb{R}^m$ is defined by
 $$
d(\mu,x)=\lim_{r\to 0}
\displaystyle\frac{log\;\mu\big(B(x,r)\big)}{log(r)},
 $$
where $B(x,r)$ denotes the closed ball of radius $r$ centered at
$x$. A probability measure $\mu$ on $\mathbb{R}^m$ is said to be
exactly dimensional if there is a constant $C$ such that
 $$
d(\mu,x)=C,\qquad \mu-a. e. x\in\mathbb{R}^m.
 $$

It is proved that the Hausdorff dimension of the measure $\mu$ is
 $$
\underline{\dim}(\mu)=C.
 $$

This result was firstly shown by Young \cite{You}. The reader can
also be referred to \cite{Fan, Mat, Pes} for the details.
\begin{definition}\label{open}
Let $m$ and $p\geq2$ be two positive integers, $X$ a nonempty
closed set of $\mathbb{R}^m$ and $\mathcal{S}=\{S_i\}_{1\leq i\leq
p}$ an IFS on $X$. Then, $\mathcal{S}$ is said to satisfy the open
set condition (OSC) if there is a non-empty, bounded and open set
$V$, such that
\begin{enumerate}
\item $\displaystyle\bigcup_{i=1}^pS_i(V)\subset V.$

\item $S_i(V)\bigcap S_j(V)=\emptyset,\quad$ if $\quad i\neq j.$
\end{enumerate}
\end{definition}

This definition allows us to recall the following result, which
will make us able to calculate the fractal dimension of $N$.

\begin{theorem}\label{rate} $[$Theorem 9.3 in \cite{Fal}$]$
Let $m$ and $p$ be two positive integers, with $p\geq2$, $X$ be a
nonempty closed set of $\mathbb{R}^m$ and
$\mathcal{S}=\{S_i\}_{1\leq i\leq p}$ be an IFS on $X$. Suppose
that, for $1\leq i\leq p$, $S_i$ is a similarity with ratio $r_i$
and attractor $F$. Suppose, also that $\mathcal{S}$ satisfies the
(OSC) and let $s$ be the unique value, such that
 $$
\displaystyle\sum_{i=1}^pr_i^s=1.
 $$

Then,
 $$
dim_H(F) =dim_B(F)=s.
 $$

In particular, if $r_1 =\ldots = r_p = r$ for some $r$, then
 $$
dim_H(F) = dim_B(F) =-\displaystyle\frac{log \;n}{log\; r}.
 $$
\end{theorem}

The reader can find a proof of the dimension formula for
self-similar sets either in \cite{Fal} or in \cite{MaU}. It is
well known \cite{Fal} that
 $$
dim_H(F) \leq dim_P(F)\leq dim_B(F).
 $$

\begin{corollary}\label{dimension}\cite{Fal}
Suppose that the conditions of Theorem \ref{rate} are satisfied.
If we have, in addition, $r_1 =\ldots = r_p = r$. Then,
 $$
\dim_H(F)=dim_P(F)=dim_B(F)=-\displaystyle\frac{log \;n}{log\; r}.
 $$
\end{corollary}

\end{appendix}


\begin{thebibliography}{99}

\bibitem{Bar} MM. Barnsley; \emph{Fractals everywhere}, Academic,
Boston (1988).

\bibitem{Fal} K. J. Falconer; \emph{Fractal Geometry}, Mathematical
Foundations and Applications. Wiley, 2nd Edition (2003).

\bibitem{Fan} A. H. Fan; \emph{Sur les dimensions de mesures},
Studia Math. (1994) 111 no.1, 1-17.

\bibitem{Fel} W. Feller; \emph{An Introduction to Probability Theory and its
Applications}, Vol. 1, 3rd edition, Wiley, New York, (1968).

\bibitem{Hut} J. J. Hutchinson; \emph{Fractals and self-similarity},
Indiana Univ. Math. J. (1981) 30, 271-280.

\bibitem{Kes} H. Kesten; \emph{An iterated logarithm law for the local
time}, Duke Math. J. (1965) 32, 447-456.

\bibitem{Mak} E. Makover and J. McGowan; \emph{An elementary proof that random
Fibonacci sequences grow exponentially}, Journal of Number Theory
(2006) 121, Issue 1, 40-44.

\bibitem{Mat} P. Mattila; \emph{Geometry of Sets and Measures in
Euclidean Spaces}, Fractals and Rectifiability, Cambridges studies
in advanced mathematics, 44, Cambridge University Press (1995).

\bibitem{MaU} R. D. Mauldin and M. Urbanski; \emph{Dimensions and
measures in iterated function systems}, Proc. London Math. Soc.
(1996) (s3-73) 105-154.

\bibitem{MaW} J. H. Ma and Z. Y. Wen; \emph{A multifractal
decomposition according to rate of returns}, Math. Nachr. (2007)
280, 888-896.

\bibitem{Neu} J. Neunhäuserer; \emph{Return of Fibonacci random
walks}, Stat. and Prob. Lett. (2017) 121, 51-53.

\bibitem{Pes} Ya. Pesin; \emph{Dimension Theory in Dynamical Systems
- contemporary views and applications}, University of Chicago
Press, Chicago (1997).

\bibitem{Pol} G. Polya; \emph{Uber eine Aufgabe der
Wahrscheinlichkeitsrechnung betreffend die Irrfahrt im
Strassennetz}, Math. Ann. (1921) 84, 149-160.

\bibitem{Rog} C. A. Rogers; \emph{Hausdorff measures}, Cambridge
Mathematical Library, Cambridge University Press (1998).

\bibitem{Vis} D. Viswanath; \emph{Random Fibonacci sequences and
the number $1:13198824...$}, Math. Comp. 69, (2000) no. 231,
1131-1155.

\bibitem{You} L.-S. Young; \emph{Dimension, entropy and Lyapunov
exponents}, Ergodic Theory Dynam. Systems 2 (1982), no. 1,
109-124.
\end{thebibliography}
\end{document}